\documentclass[3p]{elsarticle}

\usepackage[utf8]{inputenc}
\usepackage{bm}
\usepackage{algorithm}
\usepackage[noend]{algpseudocode}
\usepackage{booktabs}
\usepackage{graphicx} 
\usepackage{mathtools}
\usepackage[english]{babel}
\usepackage{amssymb,amsmath} 
\usepackage{mathrsfs}
\usepackage{bbm}

\newtheorem{theorem}{Theorem}

\newtheorem{corollary}[theorem]{Corollary}
\newtheorem{lemma}[theorem]{Lemma}
\newtheorem{definition}[theorem]{Definition}
\newtheorem{conjecture}[theorem]{Conjecture}
\newtheorem{question}{Question}

\newenvironment{proof}{\paragraph{Proof}}{\hfill$\square$\medskip}

\algnewcommand\algorithmicforeach{\textbf{for each}}
\algdef{S}[FOR]{ForEach}[1]{\algorithmicforeach\ #1\ \algorithmicdo}

\begin{document}

\title{Three aspects of the MSTCI problem}

\author[1,2]{Manuel Dubinsky\corref{cor1}}
 \ead{mdubinsky@undav.edu.ar}
 
\author[3]{César Massri \fnref{fn1}}
 
 \author[4]{Gabriel Taubin \fnref{fn2}}
 
\cortext[cor1]{Corresponding author}
\fntext[fn1]{Massri was supported by Instituto de Investigaciones Matemáticas ``Luis A. Santal\'o'', UBA, CONICET, CABA, Argentina}
\fntext[fn2]{Taubin was partially supported by NSF grant number IIS-1717355.}

\address[1]{Ingeniería en Informática, Departamento de Tecnología y Administración, Universidad Nacional de Avellaneda, Argentina}
\address[2]{Departamento de Computación, Facultad de Ciencias Exactas y Naturales, Universidad de Buenos Aires, Argentina}
\address[3]{Departamento de Matemática, Universidad de CAECE, CABA, Argentina}
\address[4]{School of Engineering, Brown University, Providence, RI, USA}

\begin{abstract}
Consider a connected graph $G$ and let $T$ be a spanning tree of $G$. 
Every edge $e \in G-T$ induces a cycle in $T \cup \{e\}$. The 
intersection of two distinct cycles is the set of edges of 
$T$ they have in common. The MSTCI problem consists in finding a spanning 
tree  that has the least number of such non-empty intersections and the  
\emph{intersection number} is the number of non-empty intersections of a 
solution. In this article we consider three aspects of the problem in a general 
context (i.e. for arbitrary connected graphs). The first presents two lower 
bounds of the intersection number. The second compares the intersection number 
of graphs that differ in one edge. The last one is an attempt to generalize a 
recent result for graphs with a universal vertex.
\end{abstract}


\maketitle

\section{Introduction}

Let $T$ be a spanning tree of $G$. Every 
edge $e \in G-T$ induces a cycle in $T \cup \{e\}$. The intersection of two 
distinct cycles is the set of edges of $T$ they have in common. 
The MSTCI problem consists in finding a spanning tree that has the least 
number of such pairwise non-empty intersections. The \emph{intersection number} 
is the number of non-empty intersections of a solution.

\medskip

The theoretical interest, as well as an application of 
the MSTCI problem, is described in 
\cite{Dubinsky2021}. A conjecture regarding graphs 
that contain a \emph{universal vertex} was recently 
proved in \cite{CHEN202219}.

\medskip

The \emph{intersection number} of $G$ is connected 
with the sparsity of the \emph{Grammian matrix} of 
a cycle matrix. Let $B = (C_1,\dots,C_\nu)$ be a 
cycle basis with corresponding cycle matrix 
$\Gamma$. The Grammian of $\Gamma$ is  
$\hat{\Gamma} = \Gamma^t \Gamma$. We will denote 
$\hat{\Gamma}$ the \emph{cycle intersection matrix} of $B$. 
It is easy to check that the $ij$-entry of $\hat{\Gamma}$ is 
0 if and only if the cycles $C_i$ and $C_j$ do not intersect 
(i.e.: they have no edges in common). 

\medskip

The MSTCI problem corresponds to the particular case of bases
 that belong to the strictly fundamental class \cite{Kavitha:2009}, 
 i.e. bases induced by a spanning tree. The 
 complexity class of this problem is unknown. In 
 this context it seems natural to consider the 
 problem of finding a lower bound 
 for the \emph{intersection number}.

 \medskip

The article presents a reflection of the MSTCI problem for 
arbitrary connected graphs. It is structured as follows: 
Section 2 sets some notation, convenient 
definitions and auxiliary results. 
Section 3 focuses on the theoretical and practical 
interest of two lower bounds of the intersection number. Section 4 
compares the intersection number of graphs that differ in one edge. 
Section 5 attempts to generalize a recent result for 
graphs with a universal vertex. Section 6 briefly describes the 
technical details of the experiments. Finally, Section 7 
collects the conclusions.

\section{Preliminaries}
\subsection{Overview}

In the first part of this section we present some 
of the terms and the notation used in the article. 
In the second part we present some lemmas 
required in the following sections.

\subsection{Notation}
Let $G=(V,E)$ be a graph and $T$ a spanning tree of $G$. 
The number of vertices and edges will be  
$|V| = n$ and $|E| = m$, resp. The number of 
cycles of a cycle 
basis is $\mu = m - n + 1$, known as the 
\emph{cyclomatic number} of $G$. The 
unique path between $u,v \in V$ in the spanning tree $T$, 
will be denoted $uTv$; max-deg(G) (resp. max-deg(T)) 
and min-deg(G) (resp. min-deg(T)) will be the maximum and minimum 
degree of a node of $G$ (resp. $T$) 

\medskip

We will 
refer to the edges $e\in T$ as 
\emph{tree-edges} and to the edges 
in $G-T$ as \emph{cycle-edges}.
Every cycle-edge $e$ induces a cycle in $T \cup 
\{e\}$, which we will call a \emph{tree-cycle}. 
We will denote $\cap_G(T)$ to 
the number of non-empty tree-cycle pairwise 
intersections w.r.t. $T$, being $\cap(G) = 
\cap_G(T)$ the \emph{intersection number} 
of $G$ in the case where $T$ is a solution of the 
MSTCI problem. 

\medskip

We shall call \emph{star} spanning tree to one that 
has one vertex that connects to all other vertices, and $K_n$ to the 
complete graph on $n$ nodes. A \emph{universal 
vertex} $u$ of $G$ is a vertex incident to 
all the other vertices, i.e. has $max-deg(G) = n - 1$. $G$ is a \emph{regular graph}  
if all its vertices have the same degree. 

\medskip

Finally we will refer to the terms \emph{vertex} 
and \emph{node} interchangeably.

\subsection{Auxiliary lemmas and definitions}

In this section we present some auxiliary lemmas 
and definitions. In the following, let $G=(V,E)$ be a 
connected graph and $T = (V, E')$ a spanning tree of $G$.

\bigskip

\begin{definition} \cite{Bondy2008}
Let $e \in E'$ be a tree-edge, consider the two 
connected components $T_1$ and $T_2$ determined 
by $T - \{e\}$, denote the 
\emph{bond} $b_e$  
as the maximal set of edges $(v,w) \in E$ such 
that $v \in T_1$ and $w \in T_2$. 
\end{definition} 

\begin{lemma}\label{lemma:bond_pairwise_intersect}
Let $e \in E'$ be an edge of $T$ and $b_e$ its 
corresponding bond, then the edges in $b_e - \{e\}$ are 
cycle-edges that determine tree-cycles that intersect pairwise.
\end{lemma}
\begin{proof}
Let $T_1, T_2$ be the two connected 
components determined by $T-\{e\}$, and let 
$(v,w) \in b_e-\{e\}$. By definition $v \in T_1$ 
and $w \in T_2$, then $vTw$ must contain $e$ and 
$vTw \cup (v,w)$ determines a tree-cycle. This 
implies that $(v,w)$ is a cycle-edge and that the 
edges in $b_e-\{e\}$ intersect pairwise.
\end{proof}

\begin{lemma}\label{lemma:cycle_edge_every_bond}
Let $f = (v,w) \in E-E'$ be a cycle-edge, then $f$ 
belongs to the bonds corresponding to the edges of 
$vTw$.
\end{lemma}
\begin{proof}
Let $e \in vTw$, consider the two connected 
components $T_1, T_2$ determined by $T - \{e\}$, 
clearly $v \in T_1$ and $w \in T_2$ then by 
definition $f \in b_e$. 
\end{proof}

\begin{corollary}\label{coro:cycle_edge_2_bonds}
Every cycle-edge belongs to at least two bonds.
\end{corollary}
\begin{proof}
Let $f = (v,w) \in E-E'$ be a cycle-edge, 
as every cycle contains at least three edges then 
$vTw$ has at least length two. By the previous 
lemma the claim follows. 
\end{proof}

Let $B_T = \{b_e\}_{e \in E'}$ be the set of bonds 
w.r.t. $T$, then -by Lemma 
\ref{lemma:bond_pairwise_intersect}- 
computing the total number of pairs of cycle-edges 
in $B_T$ overestimates $\cap_G(T)$

$$\sum_{e \in E'} \binom{|b_e| - 1}{2} \geq 
\cap_G(T)$$

The overestimation is due to redundant 
intersections. Note that if the intersection of 
two tree-cycles $c_1, c_2$ contains two or more 
tree-edges, then the corresponding pair of 
cycle-edges will be contained in more than one 
bond. The next lemmas address this problem at 
the cost of incurring in underestimation.

\begin{lemma}\label{lemma:bond_unique}
Let $B_T = \{b_e\}_{e \in E'}$ be the set of bonds 
w.r.t. $T$, then another set $\bar B_T = 
\{\bar b_e\}_{e \in E'}$ exists such that:

\begin{itemize}
	\item $\bar b_e \subseteq b_e$
	\item every edge $e \in E$ belongs to exactly one $\bar b_e$
\end{itemize}
\end{lemma}

\begin{proof}
As $E$ is finite, then the bonds of $T$ and 
$B_T$ exist and are finite. In order to define a set 
$\bar B_T$ that meets the conditions, 
suffices to remove all but one occurrence 
of the duplicated edges in the bonds of $B_T$. 


\end{proof}

This result motivates the following definition.

\begin{definition} 
Let $B_T = \{b_e\}_{e \in E'}$ be 
the set of bonds w.r.t. $T$, then a \emph{non 
redundant bond set} is a set $\bar B_T = 
\{\bar b_e\}_{e \in E'}$ such that:
\begin{itemize}
	\item $\bar b_e \subseteq b_e$
	\item every edge $e \in E$ belongs to exactly one $\bar b_e$
\end{itemize}
\end{definition}
	
\begin{lemma}\label{lemma:bound}
Let $T$ be a solution for the MSTCI problem 
w.r.t. $G$, $\bar B_T = \{\bar b_e\}_{e \in E'}$ a non 
redundant bond set and $\phi_e = |\bar b_e| - 1$ for 
every $e \in E'$, then 

\begin{itemize}
	\item $\sum_{e \in E'} \binom{\phi_e}{2} \leq \cap(G)$
	\item $\mu = \sum_{e \in E'} \phi_e$
\end{itemize}
\end{lemma}
\begin{proof}
By Lemma 
\ref{lemma:bond_pairwise_intersect} the edges in 
$\bar b_e - \{e\}$ are cycle-edges that intersect pairwise. 
There are $\phi_e =|\bar b_e| - 1$ 
cycle-edges in $\bar b_e - \{e\}$, then $\bar b_e$ 
accounts for $\binom{\phi_e}{2}$ pairwise 
intersections. Since every edge belongs to exactly 
one \emph{non redundant bond}, every tree-cycle 
pairwise intersection is counted at most 
once, then the first inequality holds. As each bond 
contains exactly $\phi_e$ cycle-edges and every 
cycle-edge belongs to exactly one bond, the second 
equation follows.
\end{proof}

\begin{lemma}\label{lemma:strictly_convex}
A quadratic function $f: \mathbb{R}^n \rightarrow 
\mathbb{R}$ defined as,

$$f(x_1, \dots, x_n) = \bm x^T \bm Q \bm x + \bm c^T \bm x + d$$
is strictly convex if and only if $Q$ is positive definite.
\end{lemma}
\begin{proof}
First, 
adding a linear form to a strictly convex function results in a strictly convex function.
Second, composing a strictly convex function with a linear change of coordinates,
gives a strictly convex function. Then, we reduced the proof to the case where $Q$ is diagonal
which is trivial.
\end{proof}

\begin{lemma}\label{lemma:matrix_is_strictly_convex}
The matrix $M_n=\bm I + \bm 
1$ 
where $\bm I$ is the 
$n \times n$ identity matrix and 
$\bm 1$ is the $n \times n$ matrix 
whose entries are all equal to 1 is positive 
definite.
\end{lemma}
\begin{proof}
By \emph{Sylvester's criterion} \cite[7.2.5]{horn:2013} a real symmetric matrix is 
positive definite if and only if all its 
leading principal minors are positive, in 
other words: the determinant of all its $k \times k$ upper left submatrices is positive.

Let $M_k=\bm I + \bm 
1$, where $1 \leq k \leq n$ be 
an upper left submatrix of $M_n$. Note that 
$M_k$ can be expressed as 
$$M_k = \bm I + \bm 1^k (\bm 1^k)^T$$
where $1^k$ is the $k$-vector whose entries 
are all equal to 1, and $\bm 1^k (\bm 1^k)^T$ is 
the outer product operation. By the 
\emph{matrix determinant lemma}, which is proved by taking determinant
to both sides of the equality,
\[
\begin{pmatrix} \mathbf{I} & 0 \\ \mathbf{v}^\textsf{T} & 1 \end{pmatrix}
\begin{pmatrix} \mathbf{I} + \mathbf{uv}^\textsf{T} & \mathbf{u} \\ 0 & 1 \end{pmatrix}
\begin{pmatrix} \mathbf{I} & 0 \\ -\mathbf{v}^\textsf{T} & 1 \end{pmatrix} =
\begin{pmatrix} \mathbf{I} & \mathbf{u} \\ 0 & 1 + \mathbf{v}^\textsf{T}\mathbf{u} \end{pmatrix},
\]
it follows that the determinant of $M_k$ is,
$$\det(M_k) = \det(\bm I + \bm 1^k (\bm 1^k)^T) = 
1 + (\bm 1^k)^T  \bm 1^k  = 1 + k > 0$$
\end{proof}

The next Theorem clarifies the case of graphs that 
contain a universal vertex in the context of the MSTCI problem.

\begin{theorem}\label{teo:universal_vertex}\cite{CHEN202219}
	If a graph $G$ admits a star spanning tree $T_s$, then

	$$\cap_G(T_s) \leq \cap_G(T)$$

	for any spanning tree $T$ of $G$.
\end{theorem}
	
\begin{corollary}\label{coro:star_solves_mstci} \cite{CHEN202219}
	If a graph $G$ admits a star spanning tree $T_s$, then $T_s$
	is a solution of the MSTCI problem w.r.t. $G$.
\end{corollary}

\begin{lemma}\label{lemma:star_formula} \cite{Dubinsky2021} 
	Let $G$ be a graph that admits a star spanning $T_s$ then 
	$$\cap(G) = \cap_G(T_s) = \sum_{u \in V-\{v\}} \binom{d(u) - 1}{2}$$
\end{lemma}

\section{Lower bounds of the intersection number}

\subsection{Overview}

A lower bound of the intersection number seems interesting 
both in theoretical and practical terms. It gives an insight in the 
MSTCI problem and can be useful for comparing algorithms. 
In this section we present two of them for an 
arbitrary connected graph. In the first part we expose 
proof of a lower bound. In the second, based on 
an experiment we conjecture an improved version.

\subsection{Proof of a lower bound}

\bigskip

\begin{theorem}\label{teo:main}
Let $G=(V,E)$ be a connected graph, then 
$$\frac{1}{2}\left(\frac{\mu^2}{n-1} - \mu\right) \leq \cap(G)$$
\end{theorem}
\begin{proof}
Let $T=(V,E')$ be a solution of the MSTCI 
problem w.r.t. $G$. By Lemma  
\ref{lemma:bond_unique} it is possible to determine 
a \emph{non redundant bond set} $\bar B_T = 
\{\bar b_e\}_{e \in E'}$. And by Lemma \ref{lemma:bound} 
the following hold:

\begin{enumerate}
\item $\sum_{e \in E'} \binom{\phi_e}{2} \leq \cap(G)$
\item $\mu = \sum_{e \in E'} \phi_e$
\end{enumerate}
where $\phi_e = |\bar b_e| - 1$ for every $e \in E'$. 
The expression $\sum_{e \in E'} 
\binom{\phi_e}{2}$ can be identified with a 
point in the image of a function $f: 
\mathbb{R}^{n-1} \rightarrow \mathbb{R}$, defined as

$$f(x_1, \dots, x_{n-1}) = \sum_{i = 1}^{n-1} \binom{x_i}{2} = \sum_{i = 1}^{n-1} \frac{x_i(x_i-1)}{2}.$$

In this setting, in order to effectively calculate a 
lower bound of the intersection number of $G$, the 
above expressions lead to the following quadratic 
optimization problem \cite[\S 9]{murty:2010},
\begin{align*}
\text{minimize} & \quad \sum_{i = 1}^{n-1} \binom{x_i}{2}\\
\text{subject to} &\quad  \mu = \sum_{i = 1}^{n-1} x_i
\end{align*}
This minimization problem can be solved by restricting a 
degree of freedom of the objective function, defining 
$\bar f: \mathbb{R}^{n-2} \rightarrow \mathbb{R}$ as
$$\bar f(x_1, \dots, x_{n-2}) = f(x_1, \dots, x_{n-2}, \mu - \sum_{i=1}^{n-2} x_i).$$

The next step is to show that $\bar f$ is a strictly 
convex function and consequently has a unique global 
minimum. The function $\bar f$ can be expressed in matrix 
form as follows,
$$\bar f(x_1, \dots, x_{n-2}) = \frac{1}{2} 
\left( \bm x^T 
(\bm I + \bm 1^{n-2 \times n-2}) \bm x - 2 \mu \bm{x}^T 
\bm 1^{n-2} + \mu^2 - \mu
\right)
$$
where $\bm x = (x_1, \dots, x_{n-2})^T$, $\bm I$ is the 
$(n-2) \times (n-2)$ identity matrix, $\bm 1^{n-2 \times n-2}$ 
and  $\bm 1^{n-2}$ are the $(n-2) \times (n-2)$  
matrix and the $n-2$ vector respectively, 
whose entries are all equal to 1.
By Lemma \ref{lemma:strictly_convex}, $\bar f$ is strictly 
convex if and only if $(\bm I + \bm 1^{n-2 \times n-2})$ is positive 
definite; this fact is proved in Lemma
\ref{lemma:matrix_is_strictly_convex}. To calculate the 
unique minimum of $\bar f$ suffices to check where 
its gradient vanishes,
$$\nabla \bar f = (\bm I + \bm 1^{n-2 \times n-2}) \bm x - \mu \bm 1^{n-2} = 0.$$
The solution of this linear system is,
$$x_i = \frac{\mu}{n-1}, \quad \forall \ 1 \leq i \leq n-2$$
and 
$$x_{n-1} = \mu - \sum_{i=1}^{n-2} x_i = \frac{\mu}{n-1}.$$
The global minimum is,
\[
\bar f\left(\frac{\mu}{n-1}, \dots, \frac{\mu}{n-1}\right) = 
\frac{1}{2}\left(\frac{\mu^2}{n-1} - \mu\right).
\]

Summarizing, the following inequalities hold,
$$\frac{1}{2}\left(\frac{\mu^2}{n-1} - \mu\right) \leq \sum_{e \in E'} \binom{\phi_e}{2} \leq \cap(G)$$
and the global minimum of $\bar f$ is a lower 
bound of the
intersection number of $G$ as claimed. 
\end{proof}

\bigskip

Next we will refer to the lower bound as

$$l_{n,m} := \frac{1}{2} \left(\frac{\mu^2}{n-1} - \mu\right)$$

\bigskip

The lower bound can be further improved to a 
strict inequality, in this way.

\begin{lemma}\label{coro:strict_inequality}
Let $G=(V,E)$ be a connected graph where 
$\mu >0$ then 
$$l_{n,m} < \cap(G)$$
\end{lemma}
\begin{proof}	
As in the previous proof, let 
$T=(V,E')$ be a solution of the MSTCI 
problem w.r.t. $G$, $\bar B_T = 
\{\bar b_e\}_{e \in E'}$ a \emph{non redundant 
bond set}, and $\bar f(x_1, \dots, x_{n-2}) = 
f(x_1, \dots, x_{n-2}, \mu - \sum_{i=1}^{n-2} 
x_i)$, where $f: \mathbb{R}^{n-1} \rightarrow 
\mathbb{R}$ is defined as,
$$f(x_1, \dots, x_{n-1}) = \sum_{i = 1}^{n-1} \binom{x_i}{2}.$$
By Theorem \ref{teo:main} these inequalities hold, 
$$l_{n,m} = \frac{1}{2}\left(\frac{\mu^2}{n-1} - \mu\right) \leq \sum_{e \in E'} \binom{\phi_e}{2} \leq \cap(G)$$
where $\phi_e = |\bar b_e| - 1$ for every $e \in E'$, and the expression in the left is 
the unique global minimum of $\bar f$ which occurs at,
$$x_i = \frac{\mu}{n-1} \quad \forall \ 1 \leq i \leq n-2$$
The previous proof also implies that,
$$x_{n-1} = \mu - \sum_{i=1}^{n-2} x_i = \frac{\mu}{n-1}$$

Suppose that equality holds,
$$\frac{1}{2}\left(\frac{\mu^2}{n-1} - \mu\right) = \sum_{e \in E'} \binom{\phi_e}{2} = \cap(G).$$
Since $\bar B_T$ is a \emph{non redundant 
bond set}, then $\phi_e = |\bar b_e| - 1$ is a 
non negative integer number for every $e \in 
E'$, implying that the first equality can 
occur if and only if 
$$\phi_e = x_i = \frac{\mu}{n-1}, \quad \forall \ 1 \leq i \leq n-1, e \in E'.$$
This, in turn, implies that $n-1$ divides $\mu$.
Under this assumption 
and that $\mu > 0$ then the elements of $\bar B_T$ 
mutually differ in $\phi_e = \mu/(n-1) \geq 1$ 
cycle-edges. 
Consider the second 
equality,
$$\sum_{e \in E'} \binom{\phi_e}{2} = \cap(G)$$
By Lemma \ref{lemma:bond_unique} and Lemma
\ref{lemma:bound}, for every cycle-edge 
the left expression in
 the equation only accounts (at most) for the 
 cycle intersections restricted to only one bond 
of $B_T$ to which it belongs. Since by 
corollary \ref{coro:cycle_edge_2_bonds} every 
cycle-edge belongs to at least two bonds, and 
they mutually differ in (at least) $\phi_e = \mu/(n-1) \geq 1$ 
cycle-edges, we conclude that,
$$\sum_{e \in E'} \binom{\phi_e}{2} < \cap(G)$$
\end{proof}
	
\subsection{Evaluation}

In this part we clarify some properties of $l_{n,m}$. 

\medskip

In order to have practical meaning $l_{n,m}$ must be a positive 
integer. This is expressed in the subsequent lemma.

\begin{lemma}\label{lemma:bound_geq_than_0}
Let $G$ be a connected graph such that $\mu > 0$
and $l_{n,m} > 0$ then $m > 2(n-1)$.
\end{lemma}
\begin{proof}
The above hypothesis implies,
\[
l_{n,m} = \frac{1}{2} \left(\frac{\mu^2}{n-1} - \mu\right) > 0 
\iff \mu^2 > (n-1) \mu 
\iff \mu = m - (n-1)> (n-1) \iff m > 2(n-1).
\]
\end{proof}

An interpretation of this fact is that: 
as Theorem \ref{teo:main} suggests an 
equidistribution of the cycle-edges among the bonds 
of the \emph{non redundant bond set}, having 
$m \leq 2(n-1)$ would lead to a situation in which 
each bond contains -at most- one cycle-edge and consequently 
no tree-cycle intersections.

\bigskip

In line with the previous result, the next lemma presents an 
upper bound condition on $m$ based on 
Theorem \ref{teo:universal_vertex}. 

\begin{lemma}\label{lemma:dense_graphs}
Let $G = (V,E)$ be a connected graph such that $m > 
n (n-2)/2$ then $G$ contains a universal vertex. 
\end{lemma}
\begin{proof}
Note that if we fix the number of the nodes $|V|= n$, the 
graph with maximum number of edges and without a universal 
vertex is a graph in which at most every node has degree 
$n-2$ and 

$$m = \frac{1}{2} \sum_{v \in V} d(v) = \frac{n (n-2)}{2}$$
\end{proof}

\begin{corollary}\label{coro:dense_graphs}
	Let $G = (V,E)$ be a connected graph such that $m > n (n-2)/2$
then $G$ contains a star spanning tree $T_s$ which is a solution of 
the MSTCI problem and 

$$\cap(G) = \cap_G(T_s) = \sum_{u \in V-\{v\}} \binom{d(u) - 1}{2}$$
\end{corollary}

\begin{proof}
	The proof follows immediately by direct application of 
	the previous lemma, Corollary \ref{coro:star_solves_mstci} 
	and Lemma \ref{lemma:star_formula}.
\end{proof}

\bigskip

The previous results show that $l_{n,m}$ has practical interest 
in the cases where $2(n-1) < m \leq n (n-2)/2$.

\bigskip

The next lemma focuses on the quality of $l_{n,m}$ in 
the sense that quantifies its underestimation in a 
particular family of graphs. 

\begin{lemma}\label{lemma:bound_regular_graph}
Let $G=(V,E)$ be a graph that contains a universal 
vertex $u \in V$ such that the subgraph $G - \{u\}$ 
is $k$-regular with $k \geq 4$ then 
$$\frac{1}{8} \leq \frac{l_{n,m}}{\cap(G)} \leq \frac{1}{4}$$
\end{lemma}
\begin{proof}
As $d(v) = k$ $\forall v \in 
V-\{u\}$, and the star spanning tree is a solution 
of the MSTCI problem then according to the 
intersection number formula,
$$\cap(G) = \sum_{v \in V-\{u\}} \binom{d(u) - 1}{2} =
\sum_{v \in V-\{u\}} \binom{k - 1}{2} = \frac{(n-1) (k-1)(k-2)}{2}$$
Taking into account that,
$$m = \frac{1}{2} \sum_{v \in V} d(v) = 
\frac{1}{2}\left((n-1) + (n-1) k\right) = \frac{(n-1)(k+1)}{2}$$
Then,
$$\mu = m - (n - 1) = \frac{(n-1)(k+1)}{2} - (n - 1) 
= \frac{(n-1)(k-1)}{2}$$

Next, we analyze the $l_{n,m}$,
$$l_{n,m} = \frac{1}{2} \left(\frac{\mu^2}{n-1} - \mu\right) = 
\frac{(n-1)(k-1)(k-3)}{8}$$
and finally we express the quotient,
$$\frac{\frac{1}{2} (\frac{\mu^2}{n-1} - 
\mu)}{\cap(G)} = \frac{k-3}{4(k-2)}$$

This implies that when $k=4$ and 
$k \rightarrow \infty$ the lower and upper bounds 
are met.
\end{proof}
\bigskip

This last result expresses two facts: 1) the 
underestimation of $l_{n,m}$ can be considerable and 
2) $l_{n,m}$ seems to perform 
better in dense graphs. In the next subsection 
we will confirm this in the case of small connected 
graphs. As the quotient considered 
in the previous lemma will be referenced below, we will 
denote it, 

$$r_G := \frac{l_{n,m}}{\cap(G)}$$

\subsection{Experimental analysis}

In this part we confirm -in line with the previous 
subsection- that, although $l_{n,m}$ is interesting 
from a theoretical perspective, it considerably 
underestimates the intersection number. We designed an 
experiment that exhaustively analyzed the set of 8-node 
connected graphs. Figure \ref{fig:lower_bound_ratio} 
expresses the results that can be summarized as follows,

\begin{itemize}
	\item $l_{n,m}$ is at most $\frac{1}{5} \cap(G)$.
	\item minimum $\cap(G)$ (red line) is 
	increasing -this will be related with a conjectural 
	result in the next subsection.
	\item maximum $\cap(G)$ (blue line) is not increasing; in 
	particular is decreasing for $m=24$ this is 
	precisely the set of graphs that contains the ``last'' graph 
	without a universal vertex (see Lemma 
	\ref{lemma:dense_graphs}). In the next section we will see 
	that this densest graph without a universal vertex achieves 
	the maximum $\cap(G)$ for $m=24$.
	\item $K_8$ has maximum intersection number; it seems 
	reasonable to consider that $K_n$ has maximum intersection 
	number for $n$-node connected graphs.
\end{itemize}
 
\begin{figure}[H]
	\centering
	\includegraphics[scale = 0.39]{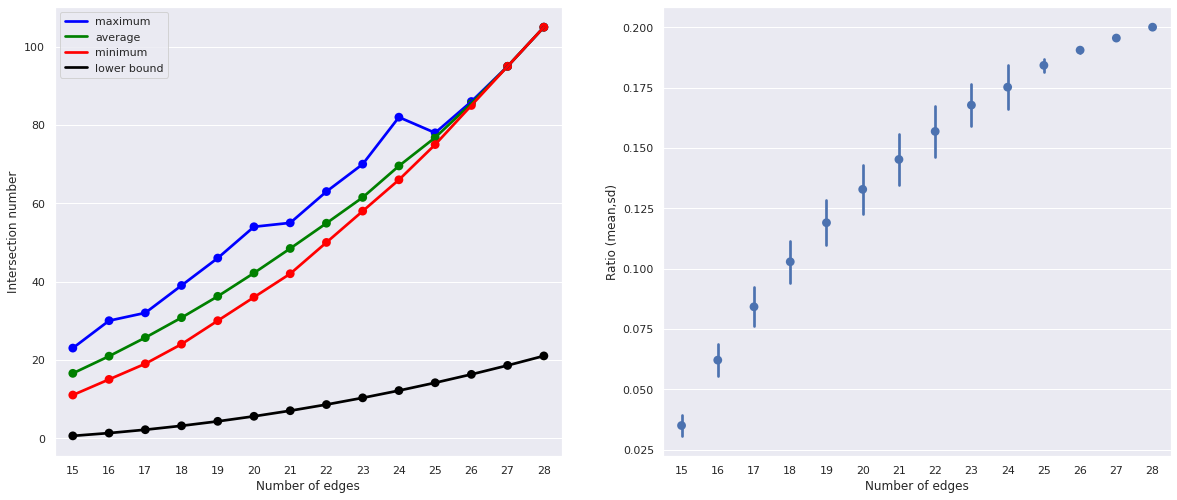}
	\caption{$8$-node connected graphs maximum $\cap(G)$, 
	minimum $\cap(G)$, 
	mean $\cap(G)$ and $l_{n,m}$ interpolating lines (left); 
	mean/standard deviation $r_G$ (right).}
	\label{fig:lower_bound_ratio}
\end{figure}

Important remark: recall that the intersection number is 
the number of pairwise intersections of a solution of 
the MSTCI problem. 
Therefore, a graph with maximum $\cap(G)$ for a fixed $m$ is a 
graph such that its intersection number is greater than or 
equal to the intersection number of any other graph of $m$
 nodes. This should not be confused with the number of 
 pairwise intersections of the problem of finding a spanning 
 tree that maximizes such number.

\subsection{A conjectural lower bound}

In this part we conjecture an improved lower bound based on 
experimental results. 

\medskip

First we focus on graphs 
$G=(V,E)$ with a universal vertex $u \in V$. Recall from 
Theorem \ref{teo:universal_vertex} that the star spanning 
tree $T_s$ is a solution of the MSTCI problem. 
Note that the sum of the degrees of the 
subgraph $G' = G - \{u\}$ verify,

$$\sum_{v \in G'} d(v) = 2 \mu$$

 The following equivalent definitions refer to those graphs with 
 a universal vertex that minimize the intersection number 
 for a fixed number of vertices and edges. The interpretation 
 is simply the equidistribution of the total degree among 
 the remaining $n-1$ vertices.

\begin{definition}
	Let $G=(V,E)$ be a graph with a universal 
	vertex $u \in V$ and $2 \mu = q (n-1) + r$ the 
	integer division of $2 \mu$ and $n-1$, we shall say 
	that $G$ is \emph{$\mu$-regular} if the degree 
	of every node $v \in V - \{u\}$ is $q+1$  or $q+2$.
\end{definition} 

\begin{definition}
	Let $G=(V,E)$ be a graph with a universal 
	vertex $u \in V$ and $2 \mu = q (n-1) + r$ the 
	integer division of $2 \mu$ and $n-1$, we shall say 
	that $G$ is \emph{$\mu$-regular}  if -except 
	for $u$- has exactly $n-1-r$ nodes with degree $q+1$ and 
	$r$ nodes with degree $q+2$.
\end{definition} 

It is simple to check the equivalence based on 
the uniqueness of $q$ and $r$. Note that for every 
$|V| = n$ and $|E| = m$ there are $\mu$-regular 
graphs; the proof is straightforward by induction on 
$\mu$. The following lemma shows that 
$\mu$-regular graphs minimize the intersection 
number of graphs with a universal vertex.

\begin{lemma}\label{lemma:mu_equidistributed}
	Let $G=(V_G,E_G)$ be a $\mu-regular$ graph and 
	$H=(V_H, E_H)$ be a non $\mu-regular$ graph  
	with a universal vertex such that $|V_G| = |V_H| = n$ and 
	$|E_G| = |E_H| = m$, then 

$$\cap(G) < \cap(H)$$
\end{lemma}

\begin{proof}
Suppose on the contrary that $\cap(H) \leq \cap(G)$. 
Let $u_H$ be a universal vertex of $H$. And -without loss of 
generality- let $H$ have a minimum intersection number; more 
precisely let $J = (V_{J}, E_{J})$ be a graph with a 
universal vertex such that $|V_{J}| = n$ and $|E_{J}| = m$ 
then the following holds,

$$\cap(H) \leq \cap(J)$$

By definition as $H$ is not $\mu$-regular then a 
node with maximum degree and a node with minimum degree 
$v_{max}, v_{min} \in V_H-\{u\}$ satisfy that 
$deg_H(v_{max}) - deg_H(v_{min}) >= 2$. Let 
$(v_{max},w) \in E_H$ be and edge such that $w \in V_H$ is 
not a neighbor of $v_{min}$. Consider the graph 
$H' = (V_H, E_H - \{(v_{max},w)\} \cup \{(v_{min},w)\})$. 
According to the hypothesis on $H$ and the formula presented
in Lemma \ref{lemma:star_formula} the following holds,

$$\cap(H) = \sum_{v \in V_H-\{u\}} \binom{d_H(v) - 1}{2} \leq \sum_{v \in V_H-\{u\}} \binom{d_{H'}(v) - 1}{2} = \cap(H')$$

As the degrees of all the nodes except for $v_{max}$ and 
$v_{min}$ coincide in $H$ and $H'$, the above inequality  
implies that,

$$\binom{d_H(v_{max}) - 1}{2} + \binom{d_H(v_{min}) - 1}{2} \leq 
\binom{d_{H'}(v_{max}) - 1}{2} + \binom{d_{H'}(v_{min}) - 1}{2}$$

Let $deg_H(v_{min}) - 1 = k$ and $deg_H(v_{max}) - 1 = k+t$, 
rewriting the inequality we have,

$$\binom{k+t}{2} + \binom{k}{2} \leq 
\binom{k + t - 1}{2} + \binom{k + 1}{2}$$

Expanding the binomials,

$$\frac{1}{2} [(k+t) (k+t-1) + k (k-1)] \leq 
\frac{1}{2} [(k+t-1) (k+t-2) + (k+1) k]$$

Which implies,

$$(k+t-1) - k \leq 0 \implies t \leq 1 \implies deg_H(v_{max}) - deg_H(v_{min}) \leq 1$$

This contradicts the hypothesis on $H$ and consequently 
proves the lemma. 
\end{proof}

Considering the general case of connected graphs 
-in an attempt to improve $l_{n,m}$- we decided to 
analyze those graphs that minimize the intersection 
number for each fixed $n$ and $m$. The experiment 
consisted of exhaustively checking the set of $7$ 
and $8$-node connected graphs. The interesting 
result was that in this general setting, 
$\mu$-regular graphs also minimize the 
intersection number; 
although not exclusively: in some cases there are 
other graphs -without a universal vertex- that 
achieve the minimum intersection number. We 
validated this fact by considering a sample of 1000 
randomly generated (by a uniform distribution)  
$9$-node connected graphs. This evaluation resulted 
positive in all cases, and consequently enables 
to formulate the following conjecture on a firm 
basis,

\begin{conjecture}\label{conj:mu_equidistributed}
	Let $G=(V_G,E_G)$ be a $\mu-regular$ graph and 
	$H=(V_H, E_H)$ be a connected graph such that 
	$|V_G| = |V_H| = n$ and $|E_G| = |E_H| = m$, then 

$$\cap(G) \leq \cap(H)$$
\end{conjecture}

Based on the intersection number formula of Lemma 
\ref{lemma:star_formula}, this conjecture implies the 
following improved -and tight- lower bound of the 
intersection number.

\begin{corollary}\label{conj:lower_bound}
	Let $G=(V,E)$ be a connected graph and 
	$2 \mu = q (n-1) + r$ the integer division of $2 \mu$ 
	and $n-1$, then 

$$\bar l_{n,m} := (n-1-r) \binom{q}{2} + r \binom{q+1}{2} = (n-1) \binom{q}{2} + q \ r\leq \cap(G)$$
\end{corollary}

\subsection{Comparison between $l_{n,m}$ and $\bar l_{n,m}$}

The following lemma expresses an improvement of $\bar l_{n,m}$
 with respect to $l_{n,m}$.

\begin{lemma}\label{lemma:bound_comparison}
	Let $m > 2(n-1)$ then, 
$$\bar l_{n,m} > 2 \ l_{n,m}$$
\end{lemma}
\begin{proof}
	The definition of $\mu$-regular graphs implies the following,

$$\frac{2 \mu - r}{n-1} = q$$

Substituting this expression in the definition of $\bar l_{n,m}$ 
we have,

$$\bar l_{n,m} = (n-1) \binom{q}{2} + q \ r  = (n-1) \frac{q \ (q-1)}{2} + q \ r\implies $$

$$\bar l_{n,m} = \frac{(n-1)}{2} \left[ \left(\frac{2 \ \mu - r}{n-1}\right)^2 - \frac{2 \ \mu - r}{n-1}\right] + r \ \frac{2 \ \mu - r}{n-1}$$

Consider the quotient,

$$\frac{\bar l_{n,m}}{l_{n,m}} = 
\frac{\bar l_{n,m}}{\frac{1}{2} \left(\frac{\mu^2}{n-1} - \mu\right)} >
\frac{\bar l_{n,m}}{\frac{1}{2} \left(\frac{\mu^2}{n-1} \right)} = 
\frac{2 (n-1) \ \bar l_{n,m}}{\mu^2} =
\frac{1}{\mu^2} \left[ (2 \mu - r)^2 - (n-1) (2 \mu - r) + 2 r (2 \mu - r) \right] =
$$

$$\frac{1}{\mu^2} \left[ 4 \mu^2 - 4 r \mu + r^2 - 2 (n-1) \mu + r (n-1) + 4 r \mu - 2 r^2 \right] = 
\frac{1}{\mu^2} \left[ 4 \mu^2 - r^2 - 2 (n-1) \mu + r (n-1) \right] =  
$$

$$
4 - \frac{r^2}{\mu^2} - \frac{2 (n-1)}{\mu} + \frac{r}{\mu^2} (n-1) = 
4 - \frac{2 (n-1)}{\mu} + \frac{r}{\mu^2} (n - 1 - r)
$$

As $0 \leq r < n-1$ and $m > 2 (n-1) \implies \mu > (n-1)$ the 
last term is equal or greater than zero then,

$$\frac{\bar l_{n,m}}{l_{n,m}} > 4 -  \frac{2 (n-1)}{n-1} = 2$$

\end{proof}

Note that $\bar l_{n,m}$ has not a lower restriction -as $l_{n,m}$-, 
meaning that it can be applied to any connected graph. Figure 
\ref{fig:lower_bound_comparison} compares 
$\bar r_G := \frac{\bar l_{n,m}}{\cap(G)}$ and $r_G$ based on 
the experiment. In practice, it seems that 
there is a better performance improvement than that of Lemma 
\ref{lemma:bound_comparison}; 
roughly speaking $\bar l_{n,m}$ is at least $\frac{1}{2}\cap(G)$ 
while $l_{n,m}$ is at most $\frac{1}{5}\cap(G)$. 
It seems that $\bar l_{n,m}$ also performs better in dense graphs.

\begin{figure}[H]
	\centering
	\includegraphics[scale = 0.5]{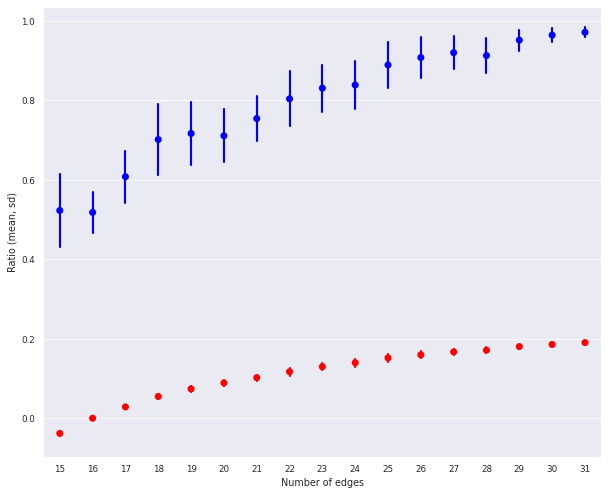}
	\caption{Comparison of mean/standard deviation of 
	$\bar r_G$ (blue) and mean/standard deviation of $r_G$ 
	(red) over a random sample of size 1000 of 
	$9$-node connected graphs.}
	\label{fig:lower_bound_comparison}
\end{figure}

Another possible comparison between $\bar l_{n,m}$ and $l_{n,m}$ 
is the following: as the graphs of the hypothesis of 
Lemma \ref{lemma:bound_regular_graph} are 
$\mu$-regular with $r=0$, then 
$\bar r_G = 1$ and consequently the bound over this 
family is tight. This is much better than the performance 
of $l_{n,m}$ ($r_G \leq 1/4$). 



\section{The intersection number of graphs that differ in one edge}

We already suggested in the previous section that the 
intersection number is not increasing w.r.t. the number of 
edges. In an attempt to understand the MSTCI problem of 
closely related graphs, it seems natural to consider the 
question of adding an edge.
 More specifically, given two connected graphs $G=(V,E)$ and 
 $G'=(V, E \cup \{e\})$ such that $e \notin E$, then could we compare 
 their intersection numbers? A first guess indicates that adding an 
 edge -and consequently a corresponding tree cycle- should increase 
 the intersection number. In this section we consider,
 
 \begin{question}\label{question:successors}
	Does a connected graph $G=(V,E)$ exist such that for any 
	graph $G'=(V, E \cup \{e\})$ ($e \notin E$), $\cap(G') 
	< \cap(G)$?
\end{question}

The first part presents a lemma that connects the 
intersection number of a graph with some of its subgraphs. 
The second introduces the definition of \emph{successor} and 
a strict condition for a connected graph to have a greater 
intersection number than all its successors.
The last part shows that in a relatively small set of 
graphs the initial guess does not hold.

\subsection{Intersection number of certain subgraphs}

The following important lemma enables comparing the 
intersection number of a graph and certain subgraphs,

\begin{lemma}\label{lemma:intersection_subgraphs}
	Let $G=(V,E)$ and $G'=(V', E')$ be connected graphs 
	such that:
\begin{itemize}
	\item $G'$ is a subgraph of $G$.
	\item $T$ is a solution of the MSTCI w.r.t. $G$.
	 \item  $T' = T \cap G'$ is a spanning tree of $G'$
	 (where $T \cap G'$ is the intersection considered 
	 as subgraphs of $G$).
\end{itemize} 
	then $\cap(G') \leq \cap(G)$

\end{lemma}
\begin{proof}
	Note that the edges $e \in E' - T'$ are cycle-edges 
	that determine the same tree-cycles w.r.t. $T'$ and $T$ 
	in $G'$ and $G$, resp. Consequently the pairwise 
	intersections between those tree-cycles are the same. Then 
	the following holds,

$$\cap(G') \leq \cap_{G'}(T') = \cap_{G}(T) - \sum_{e \in G - T - G'} \cap_T(c_e) = \cap(G) - \sum_{e \in G - T - G'} \cap_T(c_e) \leq \cap(G)$$

where $c_e$ is the tree-cycle induced by an edge $e$ and 
$\cap_T(c_e)$ are the cycle pairwise intersection of $c_e$.
\end{proof}

\begin{corollary}\label{coro:complete_graph_maximum}
	Let $G=(V,E)$ be a connected graph with a universal 
	vertex such that $|V| \leq n$ then $\cap(G) \leq 
	\cap(K_n)$
\end{corollary}
\begin{proof}
	Note that the hypothesis of the previous corollary are 
	met: $G$ is a subgraph of $K_n$, by Theorem 
	\ref{teo:universal_vertex} $T_s$ is a solution of the 
	MSTCI problem w.r.t. $K_n$, and $T_s \cap G$ is a spanning tree 
	of $G$, then

	$$\cap(G) \leq \cap(K_n)$$
	
\end{proof}

Note: the previous corollary can also be deduced by 
comparing the corresponding formulas described in
Lemma \ref{lemma:star_formula}, although the proof would 
be less conceptual.

\subsection{Intersection number of a graph and its successors}

In this section we will consider these two  
definitions.

\begin{definition}
	Let $G=(V,E)$ be a connected graph then a \emph{successor} 
	of $G$ is another graph $H = (V, E \cup \{e\})$ 
	with $e \notin E$.
\end{definition} 

And reciprocally,

\begin{definition}
	Let $H = (V, E)$ be a connected graph then a 
	\emph{predecessor} of $H$ is another connected graph 
	$G = (V, E - \{e\})$ for any edge $e \in E$.
\end{definition} 

Note that $K_n$ has no successors and if $G$ is a tree it has 
no predecessors.

\bigskip

The following lemma compares the intersection number of a 
graph with that of some of its predecessors. 

\begin{lemma}\label{lemma:decreasing_intersection_number}
	Let $H=(V,E)$ be a connected graph and $T$ a spanning 
	tree that is a solution of the MSTCI problem, 
	then for every predecessor $G=(V, E - \{e\})$ such that 
	$e \in E-T$ the following holds 

$$\cap(G) \leq \cap(H)$$
\end{lemma}
\begin{proof}
	Note that as $G$ is a subgraph of $H$ and $T$ is also a 
	spanning tree of $G$ the hypotheses of Lemma 
	\ref{lemma:intersection_subgraphs} are satisfied. 
	Consquently  $\cap(G) \leq \cap(H)$.
\end{proof}

\bigskip

Interesting in its own right is the reciprocal of this 
lemma described below.

\begin{lemma}\label{lemma:increasing_intersection_number}
	Let $G = (V,E)$ be a connected graph and 
	$H = (V, E \cup \{e\})$ a successor
	such that $\cap(H) < \cap(G)$ then for every spanning 
	tree $T$ of $H$ that is a solution of the MSTCI problem, 
	$e \in T$.
\end{lemma}

\bigskip

Based on the strong restriction imposed by Lemma 
\ref{lemma:increasing_intersection_number}, a natural aspect 
to consider is the relation of the intersection 
numbers of a graph and its successors. In other words, how 
frequent is that $\cap(G) > \cap(H)$ for a graph $G$ and 
every successor $H$?

\medskip

In the particular case that $G$ has a universal vertex the 
following lemma applies (the proof is identical to that of 
Corollary \ref{coro:complete_graph_maximum}),

\begin{lemma}
	Let $G=(V,E)$ be a connected graph with a universal 
	vertex then for every successor $H = (V, E \cup \{e\})$, 
	the following holds $\cap(G) \leq \cap(H)$.
\end{lemma}

In the particular case that $m > n(n-2)/2$ the following 
stronger lemma expresses that the intersection number is 
strictly increasing with respect to the number of edges.

\begin{lemma}
	Let $G=(V,E)$ be a connected graph such that 
	$|E| = m > n(n-2)/2$ and $|V| = n \geq 4$, then for every 
	successor $H = (V, E \cup \{e\})$, the following holds 
	$\cap(G) < \cap(H)$.
\end{lemma}
\begin{proof}
	By Lemma \ref{lemma:dense_graphs}, $G$ contains a 
	universal vertex, so the previous lemma guarantees that 
	$\cap(G) \leq \cap(H)$. To see that the inequality is 
	strict, note that both $G$ and $H$ admit star spanning 
	trees $T_s^G$ and $T_s^H$, resp., centered at the same node. 
	So Theorem \ref{teo:universal_vertex} guarantees that 

$$\cap(H) = \cap_H(T_s^H) = \cap_G(T_s^G) + \cap_{T_s^G}(c_e) 
= \cap(G) + \cap_{T_s^G}(c_e)$$

where $c_e$ is the tree-cycle induced by the edge $e=(u,v)$. 
So as to see that $\cap_{T_s^G}(c_e) > 0$ it is enough to check 
that $d(u) > 1$ or $d(v) > 1$ in $G$. Suppose that both 
$d(u) = d(v) = 1$ in $G$ and that the remaining edges have 
maximum degree: 1) the root of $T_s^G$ has degree n-1 and 2) 
the rest of the vertices have degree (n-3). Summing all,

$$m = \frac{1}{2} \sum_{v \in V} d(v) = 
\frac{1}{2} [(n-1) + (n-2) (n-3) + 2] = 
\frac{1}{2} [(n-2)^2 + 3]$$

By hypothesis $m > n(n-2)/2$, then

$$m = \frac{1}{2} [(n-2)^2 + 3] > \frac{1}{2} n(n-2)$$

which implies

$$(n-2)^2 + 3 > n(n-2) \implies - 2 (n-2) + 3 > 0 \implies 
n < \frac{7}{2}$$

\end{proof}

\subsection{Experimental analysis}

In order to measure the set of graphs that have greater 
intersection number than all its successors, we designed an 
experiment to exhaustively analyze connected graphs such 
that $5 \leq |V| \leq 8$. Table \ref{tab:small_inst} describes 
the results. 

\begin{table}[H]
	\caption{Small graphs analysis.}
	\centering
	\begin{tabular}{ccccc}
		\toprule
		$|V|$ & Non-isomorphic graphs & $\cap(G) > $   all successors\\
		\midrule
		5 & 21 & 0 \\
		6 & 112 & 0 \\
		7 & 853 & 0 \\
		8 & 11117 & 6 \\
		\bottomrule
		\label{tab:small_inst}
	\end{tabular}
\end{table}
 
Note that the set of graphs with intersection number greater 
than all its successors is relatively small. There are six 
cases for $8$-node graphs. These rare cases are described in 
Table \ref{tab:counterxample_description}. Below are 
some remarks about its columns,

\begin{itemize}
	\item The graphs are labeled incrementally 
	according to their number of edges. 
	\item The number of \emph{Succesors} of $G_i$ is 
	exactly the difference of edges between $K_8$ and 
	$G_i$, i.e.: $|E| + Succesors = 28$. Some may be isomorphic and 
	counted many times.
	\item The different intersection numbers of the 
	$Successors$ (in parenthesis in the last column) are less than 
	the number of $Successors$, this is 
	because: 1) the graphs induced by two successors 
	may be isomorphic or 2) non-isomorphic graphs 
	induced by two successors may have the same 
	intersection number.
\end{itemize}


\begin{table}[H]
	\caption{8-node graph with intersection number greater 
	than all its successors.}
	\centering
	\begin{tabular}{cccccc}
		\toprule
		Graph & $|E|$ & Successors & $\cap(G)$ (and its successors) \\
		\midrule
		$G_1$ & 19 & 9 & 46 (44, 45) \\
		$G_2$ & 20 & 8 & 54 (52, 53) \\
		$G_3$ & 22 & 6 & 60 (59) \\
		$G_4$ & 23 & 5 & 69 (66, 68) \\
		$G_5$ & 23 & 5 & 69 (66, 67) \\
		$G_6$ & 24 & 4 & 82 (75) \\
		\bottomrule
		\label{tab:counterxample_description}
	\end{tabular}
\end{table}

Figure \ref{fig:counterxample_graphs} depicts these six 
graphs. Note that $G_6$ (at the 
bottom right) is the densest graph without a universal vertex, 
it is the unique $(n-2)$-regular graph, so all of its successors 
contain a universal vertex. This graph also achieves the 
maximum intersection number for $m = 24$ as mentioned in 
the previous section. 

\smallskip

At the moment, it is not clear if these graphs are related, 
except for the fact that they are all subgraphs of $G_6$.

\begin{figure}[H]
	\centering
	\includegraphics[scale = 0.34]{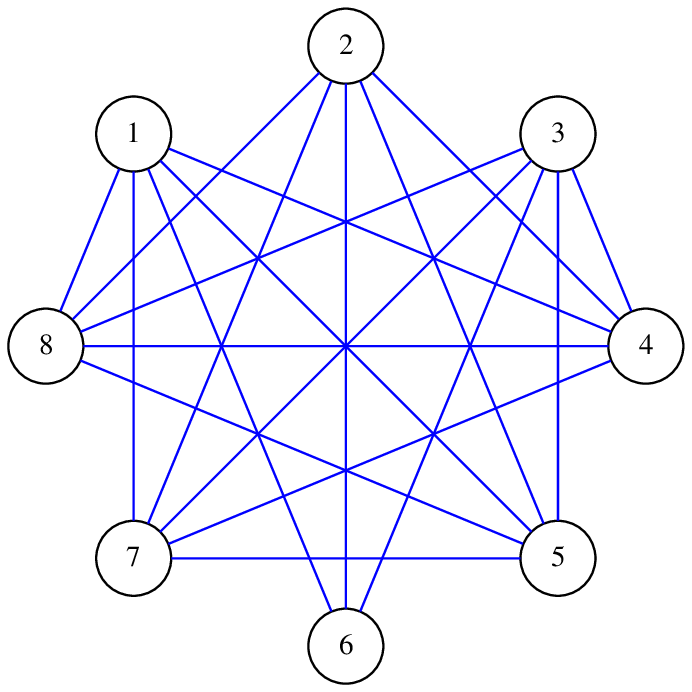}
	\includegraphics[scale = 0.34]{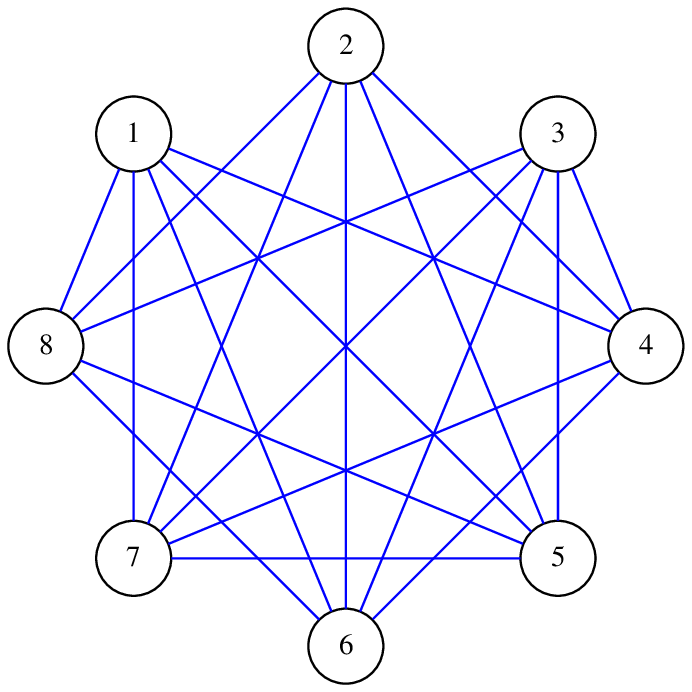}
	\includegraphics[scale = 0.34]{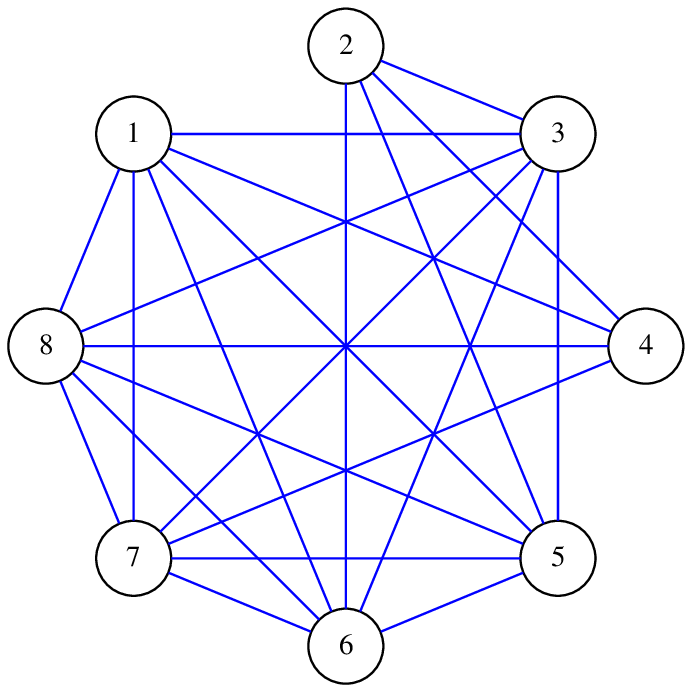} \\
	\includegraphics[scale = 0.34]{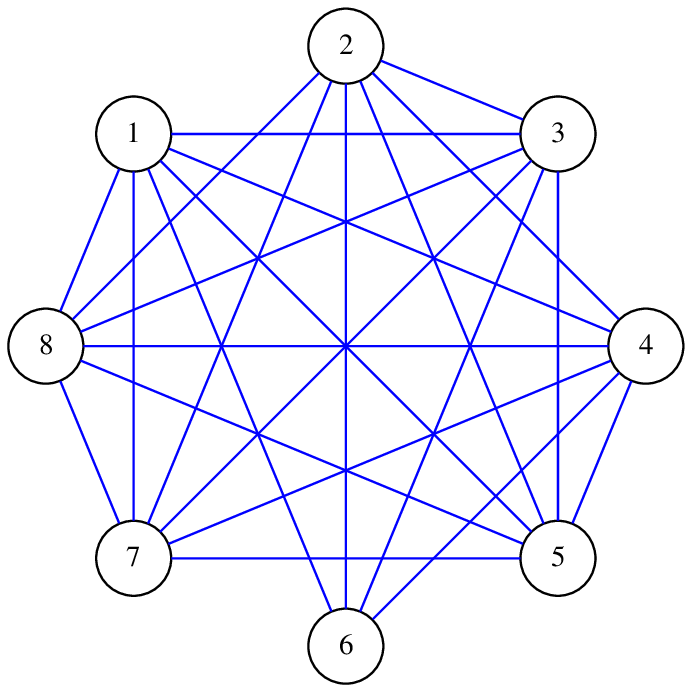}
	\includegraphics[scale = 0.34]{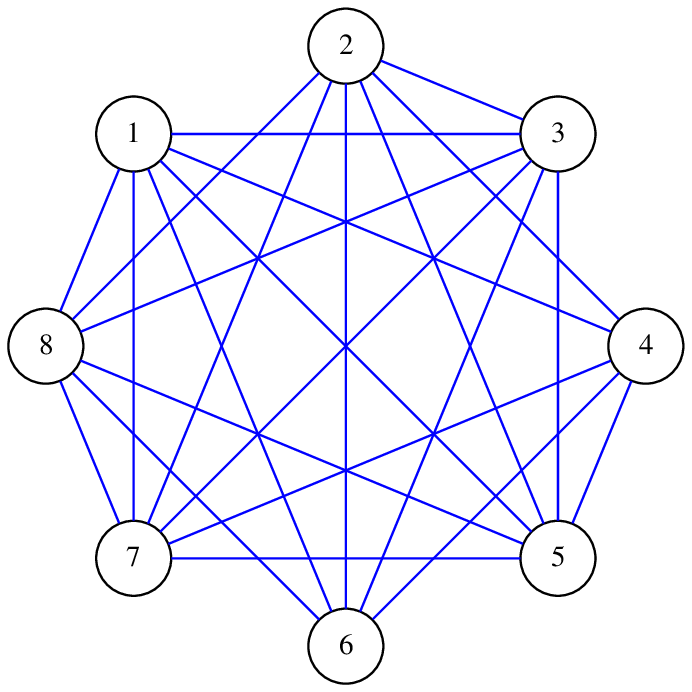}
	\includegraphics[scale = 0.34]{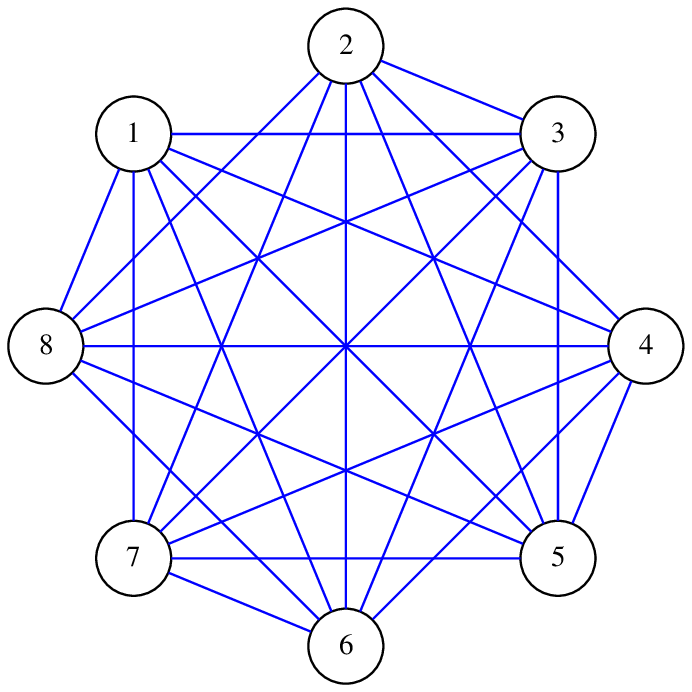}
	\caption{From top left to bottom right, 8-node graphs  
	$G_1$, $G_2$, $G_3$, $G_4$, $G_5$, $G_6$, 
	with intersection number greater than all its successors.}
	\label{fig:counterxample_graphs}
\end{figure}

\section{An attempt to generalize Theorem \ref{teo:universal_vertex}}

Theorem \ref{teo:universal_vertex} is the most general tool 
so far to understand the MSTCI problem. It would be important 
for theoretical and practical reasons to generalize it to 
an arbitrary graph. In this context we consider,

\begin{question}\label{question:degrees}
	Let $G=(V,E)$ be a connected graph, is it true that a spanning 
	tree that is a solution of the MSTCI problem always exists such 
	that $max-deg(G) = max-deg(T)$?
\end{question}

A positive answer to it would directly imply Theorem 
\ref{teo:universal_vertex}. In this section we analyze this 
problem.

\subsection{Experimental results}

As in the previous section, we exhaustively 
analyze connected graphs such that $5 \leq |V| \leq 8$. 
Table \ref{tab:small_inst2} describes the results. 
Clearly the implicit assertion of Question 
\ref{question:degrees} is false; the proportion of 
cases increases with the number of nodes.

\begin{table}[H]
	\caption{Small graphs analysis.}
	\centering
	\begin{tabular}{ccccc}
		\toprule
		$|V|$ & Non-isomorphic graphs & $max-deg(G) > max-deg(T)$ (for all MSTCI solution $T$) \\
		\midrule
		5 & 21 & 0 \\
		6 & 112 & 2 \\
		7 & 853 & 47 \\
		8 & 11117 & 1189 \\
		\bottomrule
		\label{tab:small_inst2}
	\end{tabular}
\end{table}

Figure \ref{fig:distinct_max_deg_hist} shows that the set 
of 8-node connected graphs and the set of 
cases for which the implicit assertion of Question 
\ref{question:degrees} does not hold have a similar 
distribution with 
respect to the number of edges. This suggests that the 
number of counterexamples is simply proportional and 
discourages further analysis.

\begin{figure}[H]
	\centering
	\includegraphics[scale = 0.37]{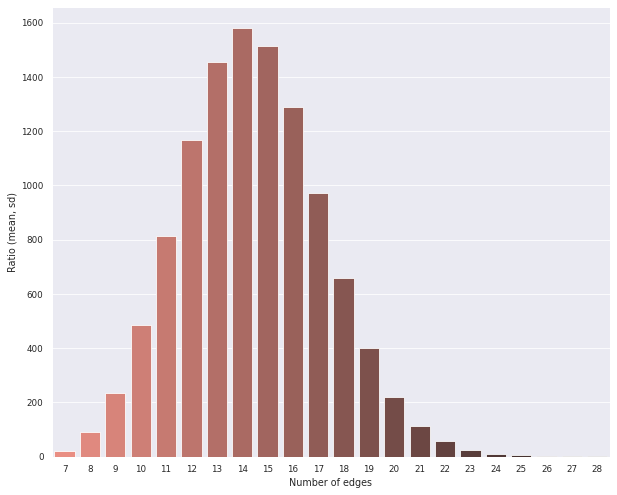}
	\includegraphics[scale = 0.37]{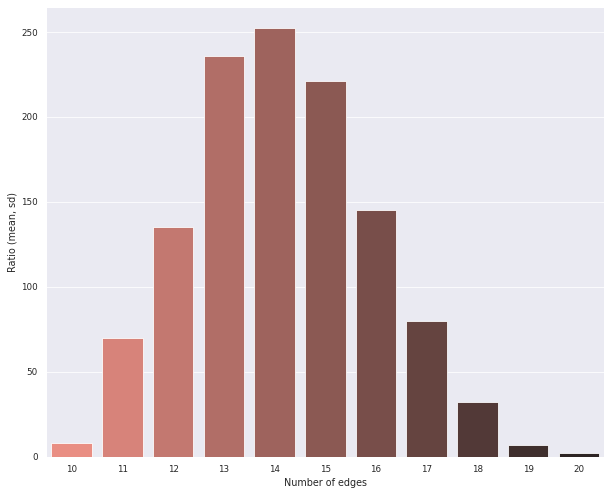}
	\caption{8-node connected graph histogram (left) and 
	counterexample histogram (right).}
	\label{fig:distinct_max_deg_hist}
\end{figure}






\section{Technical details}

A brief summary of the technical details 
of the experiments is the following,

\begin{itemize}
	\item The set of connected graphs for a given number of 
	vertices where generated with the nauty \cite{McKay2014} software.
	\item The graph figures were rendered with graphviz \cite{Ellson2004}.
	\item The programs where coded with the tinygarden java package  
	\cite{Dubinsky2021-2}. 
	\item The figures were generated with googlecolab \cite{Bisong2019}.
\end{itemize}

\section{Conclusion} 

This article considers the MSTCI problem of arbitrary 
connected graphs. In this general setting we focused on 
three different aspects. The first, presents two lower bounds 
of the intersection number of a connected graph. The proof of 
the first lower bound ($l_{n,m}$) suggests 
two structural characteristics of the MSTCI 
problem, a solution to it should be the ``best possible 
combination" of the following conditions,

\begin{itemize}
	\item Tree-cycles should have short 
	length so that each cycle-edge belongs to 
	the least number of bonds.
	\item Cycle-edges should be 
	equidistributed among bonds.
\end{itemize}

The first condition resembles the fundamental cycle basis problem 
\cite{Kavitha:2009}, and the low-stretch spanning tree 
problem \cite{Alon1995}, \cite{Elkin2005}. These similarities define 
interesting directions of research. The second tight lower bound 
($\bar l_{n,n}$) is presented in a conjectural 
form (Conjecture \ref{conj:mu_equidistributed}). 
It is based on what we denoted $\mu$-regular graphs. 
These graphs verify the previous conditions and consequently 
reinforce their importance. This last bound shows that  
graphs with a universal vertex play a fundamental role, both 
because they are well understood 
cases and they provide examples of minimum intersection 
number. It also presents a global picture of the 
quadratic surface of the minimum intersection 
number as a function of the number of nodes 
and edges.

\bigskip

The second aspect of the article constitutes a first step 
towards understanding the relation of a connected graph $G$ and 
the set of graphs that contain one edge more than $G$, which we 
denoted its \emph{successors}. We showed that there is a 
strong restriction for a connected graph to have an 
intersection number greater than all of them. Experimental results 
showed that those graphs actually exist although they seem 
to be infrequent. The most notable example is the $8$-node 
$6$-regular graph. This fact seems to be general, and can be 
expressed in conjectural terms as: for $n$ even the 
$(n-2)$-regular graph has intersection number greater than 
its successor (note that in this case there is just one). 
These graphs are the densest without a 
universal vertex. They are good examples to 
investigate this singular behavior, 
and more generally they are a good starting point to 
understand relevant structural properties.

\bigskip

Theorem \ref{teo:universal_vertex} is the most important tool 
so far to understand the MSTCI problem. It implies that the densest 
cases are known (Lemma \ref{lemma:dense_graphs}) and 
that the efforts should concentrate in cases where 
$m \leq n(n-2)/2$. In the last part of the article we explore 
the possibility of a natural generalization of this theorem. More 
specifically, we considered the fact that if for every 
connected graph $G$ a solution $T$ containing a node of maximal 
degree should exist, i.e. $max-deg(G) = max-deg(T)$.
Experimental results determined a negative answer. 
Counterexamples have the same distribution with 
respect to the number edges for a fixed number of nodes than 
the set of connected graphs which indicates that they are simply 
proportional and consequently discourage any further analysis 
in this direction.




\bibliographystyle{elsarticle-num-names}
\bibliography{mstci}

\end{document}